\documentclass[a4paper,11pt]{amsart}
\usepackage{amsmath,amssymb}
\usepackage[latin1]{inputenc}
\usepackage{fancyhdr}
\usepackage{verbatim}

\newcommand{\field}[1]{\mathbb{#1}}

\newcommand{\N}{\field{N}}

\newcommand{\Y}{\mathrm{Y}}

\newcommand\enn{\mathbb N}
\newcommand\err{\mathbb R}
\newcommand{\loglike}[1]{\mathop{\mathrm {#1}}\nolimits}

\newcommand{\diam}{\loglike{diam}}

\newcommand{\eps}{\varepsilon}

\newcommand{\ncconv}{\overline{\mbox{conv}}}

\newcommand{\HB}{\text{H{\kern -0.35em}B}}
\newcommand{\vertiii}[1]{{\left\vert\kern-0.25ex\left\vert\kern-0.25ex\left\vert#1\right\vert\kern-0.25ex\right\vert\kern-0.25ex\right\vert}}

\DeclareMathOperator{\spann}{span}

\newcounter{abcd}
{\setcounter{abcd}{0}
\begin{list}%
{{\rm (\alph{abcd})}} 
{\usecounter{abcd}
\parsep=\parskip
\topsep=1pt plus 2pt minus 1pt
\itemsep=1pt plus 2pt minus 1pt
\leftmargin=3\baselineskip \labelsep=.6\baselineskip
\labelwidth=2.4\baselineskip
\rightmargin 0pt}%
}%
{\end{list}}

\title{Diameter 2 properties and convexity}

\author[T.~A.~Abrahamsen]{Trond A.~Abrahamsen}
\address[T.~A.~Abrahamsen]{Department of Mathematics, University of  Agder, Postbox 422 4604 Kristiansand, Norway.} 
\email{trond.a.abrahamsen@uia.no}
\urladdr{http://home.uia.no/trondaa/index.php3}

\author[P.~H{\'a}jek]{Peter H{\'a}jek}
\address[P.~H{\'a}jek]{Mathematical Institute, Czech Academy of
  Science, \v{Z}itn{\'a} 25, 115 67 Praha 1, Czech Republic, and
  Department of Mathematics, Faculty of Electrical Engineering, Czech
  Technical University in Prague, Zikova 4, 160 00, Prague}
\email{hajek@math.cas.cz}

\author[O.~Nygaard]{Olav Nygaard}
\address[O.~Nygaard]{Department of Mathematics, University of Agder, Servicebox
  422, 4604 Kristiansand, Norway}
\email{Olav.Nygaard@uia.no}
\urladdr{http://home.hia.no/$~$olavn/}

\author[J.~Talponen]{Jarno Talponen}
\address[J.~Talponen]{Department of
  Physics and Mathematics, University of Eastern Finland, Box 111,
  FI-80101 Joensuu, Finland} 
\email{talponen@iki.fi}

\author[S.~Troyanski]{Stanimir Troyanski}
\address[S.~Troyanski]{Departamento de Matem{\'a}ticas, Universidad de Murcia,
Campus de Espinardo, 30100 Espinardo (Murcia), Spain, and Institute of
Mathematics and Informatics, Bulgarian Academy of Science, bl.8,
acad. G. Bonchev str. 1113 Sofia, Bulgaria} 
\email{stroya@um.es}

\thanks{The second author was financially supported by GA CR
  P201/11/0345 and RVO:   67985840. The fourth author was financially
  supported by Finnish Cultural Foundation, V\"{a}is\"{a}l\"{a}
  Foundation, and Academy of Finland Project \#268009. The fifth
  author was partially supported by MTM2014-54182-P and the
  Bulgarian National Scientific Fund under Grant DFNI-I02/10.}


\newtheorem{thm}{Theorem}[section]
\newtheorem{prop}[thm]{Proposition}
\newtheorem{lem}[thm]{Lemma}
\newtheorem{cor}[thm]{Corollary}

\theoremstyle{definition}

\newtheorem{defn}[thm]{Definition}

\newtheorem{quest}{Question}
\theoremstyle{remark}

\begin{document}

\begin{abstract}
  We present an equivalent midpoint locally uniformly rotund (MLUR)
  renorming $X$ of $C[0,1]$ on which every weakly compact projection $P$
  satisfies the equation $\|I-P\| = 1+\|P\|$ ($I$ is the
  identity operator on $X$). As a consequence we obtain an MLUR space $X$
  with the properties D2P, that every non-empty relatively weakly open subset
  of its unit ball $B_X$ has diameter 2, and the LD2P+, that for
  every slice of $B_X$ and every norm 1 element $x$ inside the
  slice there is another element $y$ inside the slice of distance as
  close to 2 from $x$ as desired. An example of an MLUR space
  with the D2P, the LD2P+, and with convex combinations of slices of arbitrary
  small diameter is also given. 
\end{abstract}

\maketitle

\section{Introduction}\label{sec:intro}
Let $X$ be a Banach space. We say that $X$ (or its norm $\|\cdot\|$)
is \emph{midpoint locally uniformly rotund [MLUR] (resp. weak midpoint
  locally uniformly rotund [weak MLUR])} if every $x$ in the
unit sphere $S_X$ of $X$ is a strongly extreme point (resp.
strongly extreme point in the weak topology), i.e.  for every
sequence $(x_n)$ in $X$, we have that $x_n \to 0$ in norm (resp. $x_n
\to 0$ weakly) whenever $\|x \pm x_n\| \to 1$. 

Let $x^* \in S_{X^*}$ and $\eps >0$. By a slice of the unit
ball $B_X$ of $X$ we mean a set of the form \[S(x^*,\eps):=\{x \in B_X:
x^*(x) > 1 - \eps\}.\] 

Over the latest 15 years quite much has been discovered concerning
Banach spaces with various kinds of diameter 2 properties (see
e.g. \cite{MR2002e:46057}, \cite{MR3098474}, \cite{HL2}, \cite{HLP},
\cite{BGLPRZ4}, \cite{BGLPRZ1} to mention a few). 

\begin{defn} \label{defn:diam2p} A Banach space $X$ has the
  \begin{enumerate}
  \item [a)] {\it local diameter 2 property} (LD2P) if every slice of $B_X$ has
    diameter 2.
  \item [b)] {\it diameter 2 property} (D2P) if every
    non-empty relatively weakly open subset of $B_X$ has diameter 2.
  \item [c)] {\it strong diameter 2 property} (SD2P) if every finite convex
    combination of slices of $B_X$ has diameter 2.
  \end{enumerate}
\end{defn}

By \cite[Lemma~II.1~p.~26]{MR912637} c) implies b) and of course b)
implies a). Non of the reverse implications hold (see
\cite[Theorem~2.4]{BGLPRZ4} and \cite[Theorem~1]{HL2} or
\cite[Theorem~3.2]{ACBGLPR}). However, note Proposition
\ref{prop:prop0} below which is an immediate consequence of Choquet's
lemma (see e.g. \cite[Lemma~3.69 p.~111]{MR2766381}).

\begin{lem}[Choquet]\label{thm:choquet}
  Let $C$ be a compact convex set in a locally convex
  space $X$. Then for every $x \in \mbox{ext}(C)$, the extreme points
  in $C$, the slices of $C$ containing $x$ form a neighborhood base
  of $x$ in the relative topology of $C$. 
\end{lem}

\begin{prop}\label{prop:prop0} 
  If $X$ is weak MLUR then the LD2P implies the D2P. 
\end{prop} 

\begin{proof}
  Simply recall that the points in $B_X$ which are strongly extreme in
  the weak topology are exactly the extreme points which continue to be
  extreme in $B_{X^{**}}$ (see e.g. \cite{MR1070518}) and then use
  Lemma \ref{thm:choquet} on $B_{X^{**}}$ given the weak$^*$ topology.  
\end{proof}

It is not evident that weak MLUR spaces with the LD2P
exist, but indeed they do. The quotient $C(\mathbb{T})/A$, where
$C(\mathbb{T})$ is the space of continuous functions on the complex unit circle
$\mathbb{T}$ and where $A$ is the disc algebra, is such an
example. Another example can be constructed as follows: Let $\Phi$ be
a function on $c_0$, the space of real valued sequences which converges to
$0$, defined by $\Phi(x_n) = \sum_{n=1}^\infty x^{2n}_n$. Define then a
norm on $c_0$ by $|x|=\inf\{\lambda >
0:\Phi(x/\lambda) \le 1\}$ for every $x \in c_0$. Then the space $(c_0,
|\cdot|)$ can be seen to be weak MLUR and to have the LD2P. 

The two examples mentioned
motivates the following question which we will address in this paper:
How rotund can a Banach space be and still have diameter 2 properties?
In \cite[Remarks 4) p.~286]{HaRao} it is pointed out that
$C(\mathbb{T})/A$ is M-embedded and that its dual norm is smooth (see
also \cite{Luck} and \cite[p.~167]{HWW}). Recall that $X$ is
M-embedded provided we can write $X^{***} = X^* \oplus_1 X^\perp$
where $X^\perp \subset X^{***}$ is the annihilator of $X$ (a good
source for the available theory of M-embedded spaces is
the book \cite{HWW}). It is well known that M-embedded spaces have the SD2P
\cite{MR3098474}, and so $C(\mathbb{T})/A$ actually furnishes an
example of a weak MLUR space with the SD2P. On can prove also that the
space $(c_0,|\cdot|)$ mentioned above have the same
properties. For still more examples see \cite{Dirk5}. The unit ball of
an M-embedded space cannot, however, contain strongly extreme points
\cite{HWW}, so no MLUR M-embedded space exists. Still, one can ask if there exists an MLUR space with the LD2P (=D2P in this case). Until now, no such example has been
known. But, in Section \ref{sec:mlur-d2p} of this paper we construct an
equivalent MLUR renorming $X$ of $C[0,1]$ for which every slice $S$ of
$B_X$ and every $x \in S \cap S_X$ there exists $y \in S$ of distance
as close to 2 from $x$ as we want, i.e. $X$ has the local diameter 2
property + (LD2P+). In particular this renorming has the LD2P and thus
the D2P as it is MLUR. Using this renorming we also construct in
Section 2 an example of an MLUR space which has the D2P, the LD2P+,
and has convex combinations of slices with arbitrary small diameter.   

In Section \ref{sec:ld2p+} we characterize the property LD2P+. We show
that a space $X$ has the LD2P+ if and only if the dual has the
weak$^*$ LD2P+ if and only if the equation $\|I - P\| = 1 + \|P\|$
holds for every weakly compact projection on $P$ on $X$ . It is also
proved that the LD2P+ is inherited by ai-ideals (see p.~11 for the
definition of this concept). 

In Section 4 we show that if a Banach lattice contains a strongly
extreme point $x$ which can be approximated by finite rank projections
at $x$, then $x$ is already a denting point. From this we see that it might
not be so easy to construct an equivalent MLUR norm on $c_0$ with the
D2P.

Section 5 contains a list of open questions.

The notation we use is mostly standard and is, if considered necessary,
explained as the text proceeds. 

\section{MLUR renormings of $C[0,1]$ with the D2P}\label{sec:mlur-d2p}

Let $D=(D_n)_{n=1}^\infty$ be a base of neighborhoods in $[0,1]$. For
each $x \in C[0,1]$ put $\|x\|_n=\sup_{d \in D_n}|x(d)|$ and note that
each $\|\cdot\|_n$ defines a semi-norm on $C[0,1]$. Now define a 
norm on $C[0,1]$ by
\[\|x\|_D:=(\sum_{n=1}^\infty2^{-n}\|x\|_n^2)^{1/2}.\]
By compactness there exists $b > 0$ such that
$b\|x\|_{\infty} \le \|x\|_D \le \|x\|_{\infty}$ so the norm $\|\cdot\|_D$ is
equivalent to the max-norm $\|\cdot\|_\infty$ on $C[0,1]$. The idea to
introduce this norm goes back to \cite{MR2474800}. 

In what follows let $X_D = (C[0,1],\|\cdot\|_D)$.

\begin{prop}\label{prop:mlurX} 
  For any base $D=(D_n)_{n=1}^\infty$ of neighborhoods in $[0,1]$ the
  space $X_D$ is MLUR.
\end{prop}

\begin{proof} 
Let $x$ and $(y_k)_{k=1}^\infty$ be such that $\lim_{k\to\infty}\|x\pm y_k\|_D=\|x\|_D$. We will show that $\|y_k\|_D\to 0$ to establish that $\|\cdot\|_D$ is MLUR.
By a convexity argument (see e.g. \cite[Fact II.~2.~3]{MR1211634}) we have
\begin{equation}\label{eq:likn1}
   \lim_{k\to\infty}\|x\pm y_k\|_n=\|x\|_n, \hspace*{1cm}n=1,2,\ldots
\end{equation}
Let $\varepsilon>0$. We will first make three simple observations:
\begin{enumerate}
  \item[a)] By uniform continuity of $x$, we can find
   $\delta=\delta(\varepsilon)>0$ such that the oscillation over $A$,
   $\sup_{t,s\in A}(x(t)-x(s))$, is less than $\varepsilon$ whenever
   $A\subset [0,1]$ is of length less then $\delta$.
 \item[b)] Proceeding with the $\delta$ above, since
   $(D_n)_{n=1}^\infty$ is a base for the topology on the
   compact space $[0,1]$, there is a finite subset
   $M\subset\mathbb{N}$ such that $\cup_{m\in M}D_m$ covers
   $[0,1]$ and the length of any $D_m, m\in M$ is less than
   $\delta$.
  \item[c)] Proceeding with $M$ and $\delta$ as above, having in
    mind (\ref{eq:likn1}) which of course is true for each $m\in M$, we
    can find $K\in\mathbb{N}$ such that $\|x\pm y_k\|_m\leq
    \|x\|_m+\varepsilon$ whenever $k\geq K$ and $m\in M$.
\end{enumerate}

We are now ready to finish the proof. To this end, let $t_0$ be an
arbitrary point in $[0,1]$. We will show that $|y_k(t_0)|\leq
2\varepsilon$ for $k\geq K$. Choose $\sigma_k$ from $\{-1,1\}$ such
that \[|x(t_0)+\sigma_k y_k(t_0)|=|x(t_0)|+|y_k(t_0)|.\]
Since $\cup_{m\in M} D_m$ covers $[0,1]$ there is $m'\in M$ such that
$t_0\in D_{m'}$. Now remember that the length of all the $D_m$'s are $<\delta$, such that the oscillation of $x$ over $D_{m'}$ is less than $\varepsilon$. We get
\[\begin{array}{ll} |y_k(t_0)|&=|x(t_0)+\sigma_k(t_0)y_k(t_0)|-|x(t_0)|\\[2mm]
&\leq \sup_{t\in D_{m'}}|x(t)+\sigma_k(t)y_k(t)|-|x(t_0)|\\[2mm]
&\leq \|x+\sigma_k y_k\|_{m'} - (\|x\|_{m'}-\varepsilon)\\[2mm]
&\leq \|x\|_{m'}+\varepsilon -\|x\|_{m'} +\varepsilon=2\varepsilon,
\end{array}\]
provided $k\geq K$.
\end{proof}

\begin{prop}\label{prop:ld2pX}
  For any base $D=(D_n)_{n=1}^\infty$ of neighborhoods in
  $[0,1]$ the space $X_D$ has the LD2P+. 
\end{prop}

\begin{proof}
We know that the dual of
 $X$ is isomorphic to
 $\mbox{rca}[0,1]$, the space of regular and countably additive Borel
 measures on $[0,1]$. Let $\lambda \in \mbox{rca}[0,1]$ be the
 Lebesgue measure. By Lebesgue's decomposition theorem, any
 measure $m \in \mbox{rca}[0,1]$ can be decomposed as $m = \mu + \nu$,
 where $\mu$ is absolutely continuous with respect to  $\lambda$ and
 $\nu$ and $\lambda$ are singular. 

  Now, let $m \in S_{X^*}$, $\eps > 0$, and denote by $S$ the
  slice \[\{x \in B_X: \int_{[0,1]} x \,dm > 1 - \eps\}.\] Let $x \in
  S$ and find $1- \|x\| \le \delta < \eps$ and $N \in \enn$ such
  that \[(\sum_{n=1}^N 2^{-n}\|x\|_n^2)^{1/2} > 1 - \delta > 1 - \eps.\] 
  
  There exist open intervals $E_n=(r_n,t_n)$ inside 
  $D_n$ with $s_n=\frac{r_n + t_n}{2}$ such that

  \begin{enumerate}
    \item[a)]$E_i \cap E_j = \emptyset$ for every $i \not= j$, 
    \item[b)]$(\sum_{n=1}^N 2^{-n}|x(e_n)|^2)^{1/2} > 1 - \delta$ whenever
      $e_n \in E_n$,
    \item[c)] $\nu(\{s_n\}) = 0$ for every $1 \le n \le N$,
    \item[d)]$b^{-1}\sum_{n=1}^N m(E_n) < \eta$ where $E=\cup_{n=1}^N
      E_n$, and $\int_{[0,1] \setminus E} x\,dm - \eta > 1- \eps,$    
  \end{enumerate}

Now, define a continuous function $y$ on $[0,1]$ by letting $y(r_n)=x(r_n)$, $y(s_n) =
-x(s_n)$, $y(t_n)= x(t_n)$, linear on
$(r_n, s_n)$ and $(s_n,t_n)$, and otherwise equal to
$x$. Then $y \in X$ with $\sup_{d \in E_n}|y(d)| \le \sup_{d \in
  E_n}|x(d)|$ and $y(d)=x(d)$ for every $d \in [0,1] \setminus
E$. Therefore $\|y\| \le \|x\| \le 1$. Moreover, we have
\begin{align*}
  \int_{[0,1]}y \,dm &= \int_{[0,1]\setminus E}y \,dm + \int_{E}y \,dm\\
                             &\ge \int_{[0,1]\setminus E}x \,dm - \sum_{n=1}^N
                               b^{-1} m(E_n) > \int_{[0,1]\setminus E}x \,dm -
                               \eta > 1 - \eps,
\end{align*}
 and
 \begin{align*}
   \|x-y\| &\ge (\sum_{n=1}^N 2^{-n}\|x-y\|_n^2)^{1/2} \\ 
                & \ge (\sum_{n=1}^N 2^{-n}|x(s_n) - y(s_n)|^2)^{1/2} =2
                  (\sum_{n=1}^N 2^{-n}|x(s_n)|^2)^{1/2} > 2 - 2\delta.          
 \end{align*}
\end{proof}

From the propositions \ref{prop:mlurX}, \ref{prop:ld2pX}, and
\ref{prop:prop0} we obtain the following result.

\begin{thm}\label{thm:mlur-d2p}
  For any base $D=(D_n)_{n=1}^\infty$ of neighborhoods in
  $[0,1]$ the space $X_D$ is MLUR and has
  the D2P and the LD2P+.
\end{thm}

In \cite{HLP} dual characterizations of the diameter 2 properties in
Definition \ref{defn:diam2p} were obtained. To formulate these we need
to introduce some concepts.

\begin{defn}
  For a Banach space $X$ we say that (the norm on) $X$ is
  \begin{enumerate}
    \item [a)]  \emph{locally octahedral}) if for every $\eps >0$ and
      every $x \in S_X$ there exists $y \in S_X$ such that $\|x \pm y\| > 2 - \eps$. 
    \item[b)] \emph{octahedral} if for every $\eps >0$ and every
      finite set of points $(x_i)_{i=1}^n \subset S_X$ there exists $y
      \in S_X$ such that $\|x_i + y\| > 2 - \eps$ for every $1 \le i
      \le n$.
  \end{enumerate}
\end{defn}
  
  For a Banach space $X$, $x \in S_X$, and $\eps >0$ we mean
  by a \emph{weak$^*$-slice} of $B_{X^*}$ a set of the from \[S(x,\eps):=\{x^* \in
  B_{X^*}: x^*(x) > 1 - \eps\}.\]

 \begin{defn} \label{defn:w*diam2p} A dual Banach space $X^*$ has the
  \begin{enumerate}
  \item [a)] {\it weak$^*$-local diameter 2 property (weak$^*$-LD2P)}
    if every weak$^*$-slice of $B_{X^*}$ has diameter 2.
  \item [b)] {\it weak$^*$-strong diameter 2 property (weak$^*$-SD2P)}
    if every finite convex combination of weak$^*$ slices of $B_{X^*}$ has diameter 2.
  \end{enumerate}
\end{defn}

\begin{thm}{\cite[Theorems~3.1, 3.3, and 3.5]{HLP}}\label{thm:diam2-char}
  For a Banach space $X$ we have
  \begin{enumerate}
    \item [a)] $X$ is locally octahedral $\Leftrightarrow$ $X^*$ has
      the weak$^*$-LD2P.
    \item [b)] $X$ is octahedral $\Leftrightarrow$ $X^*$ has the weak$^*$-SD2P.
  \end{enumerate}
\end{thm}

It follows from Theorem \ref{thm:mlur-d2p}, Theorem
\ref{prop:ld2p+-char} below, and Theorem \ref{thm:diam2-char} that for any
base $D=(D_n)_{n=1}^\infty$ of neighborhoods in $[0,1]$ the space
$X_D$ is locally octahedral. However, every such
space $X_D$ fails to be octahedral. To see this we will use the following lemma. 

\begin{lem}\label{lemma:n-norm}
  Let $u$ and $v$ be continuous functions on the unit
  interval. Suppose $\|u\|_n = \|v\|_n$ for every  $n \in
  \enn$. Then \[|u(t)| = |v(t)| \text{ for every } t \in [0,1].\] 
\end{lem}
 
\begin{proof}
  Let $\eps, \delta >0$ such that
  \begin{equation}
    \label{eq:6}
    |u(s') -u(s'')| < \eps \text{ and }   |v(s')-v(s'')| <\eps
  \end{equation}
   whenever $|s' - s''|<\delta$. Fix $t \in [0,1]$.  There exists $n
   \in \enn$ such that $t$ belongs to $D_n$ and $\diam(D_n) < \delta$. Now find
   $t', t''$ in $D_n$ such that $|\|u\|_n - |u(t')|| < \eps$ and $|\|v\|_n -
   |v(t'')|| < \eps$. Then $||u(t')| - |v(t'')|| < 2\eps$, and thus by (\ref{eq:6}) we have
   $||u(t)| - |v(t)|| < 4\eps$.
\end{proof}

\begin{prop}
  For any base $D=(D_n)_{n=1}^\infty$ of neighborhoods in
  $[0,1]$ the space $X_D$ fails to be octahedral.
\end{prop}

\begin{proof}
  Choose two different non negative norm 1 functions $u$ and $v$ in
  $X_D$. Assume there exists a sequence $(y_k)_{k=1}^\infty \subset S_X$ such
 that
 \begin{equation}
   \label{eq:5}
    \lim_{k \to \infty}\|u + y_k\|_D =  2 \text{ and } \lim_{k \to
      \infty}\|v+ y_k\|_D = 2.   
 \end{equation}
Using (\ref{eq:5}) and \cite[Fact II.~2.~3]{MR1211634} we have for
every $n \in \mathbb N$ that \[\|u\|_n = \lim_k\|y_k\|_n
=\|v\|_n.\] Now we get from
Lemma \ref{lemma:n-norm} a contradiction as $u$ and $v$ are non
negative and different.
\end{proof}

The final part of this section will be devoted to showing that 
there exists a Banach which is MLUR, has the D2P, the LD2P+, and has
convex combinations of slices with arbitrarily small diameter. First
we will show that for any given $\delta>0$ there exists a base $D =
(D_n)_{n=1}^\infty$ of neighborhoods in $[0,1]$ for which $B_{X_D}$ 
contains convex combinations of slices with diameter $< \delta$. In
that respect the following lemma will come in to use.

Let $t \in [0,1]$. Put $J(t) = \{n: t \in D_n\}$ and let $w(t) =\sum_{n \in J(t)}
 2^{-n}$. 

\begin{lem}\label{fact:dirac-norm}
  Let $D=(D_n)_{n=1}^\infty$ be a base of neighborhoods in
  $[0,1]$, $t \in [0,1]$, and $\delta_t$ the point
  measure in $X^*_D$. If $\overline{D_n} \cap
  \{t\} = \emptyset$ for every $n \not\in J(t)$,
  then \[\|\delta_t\|^*_D =\frac{1}{\sqrt{w(t)}},\] where
  $\|\cdot\|^*_{D}$ is the norm in $X^*_D.$
\end{lem}

\begin{proof}
  Let $x \in X_D$ with norm $1$. Then
  \begin{align*}
    1  = \sum_{n=1}^\infty 2^{-n}\|x\|_n^2 \ge \sum_{n \in J(t)}
    2^{-n} |x(t)|^2 = w(t) |\delta_t(x)|^2.
   \end{align*}
  Thus $\|\delta_t\|^*_D \le \frac{1}{\sqrt{w(t)}}.$ Moreover, by the
  assumptions it is always possible to find for
  $i \not\in J(t)$ an open set which contains
  $t$ and which does not intersect $\overline{D_i}$. Thus we can
  always find an $x \in S_{X_D}$ which takes its maximum value at
  $t$ and which is zero on $D_i$. From this it follows that for any
  $\eps >0$ we can
  find $x \in S_{X_D}$ which takes its maximum value at $t$ such that
  $\sum_{n \not \in J(t)}2^{-n}\|x\|_n^2 < \eps$. From the inequality 
  \begin{align*}
    1 = \|x\|_D &= \sum_{n \in J(t)}2^{-n}x(t)^2 + \sum_{n \not\in
                  J(t)}2^{-n}\|x\|^2_n\\
                      & < \sum_{n \in J(t)}2^{-n}x(t)^2 + \eps
  \end{align*}
  we get that $\delta_t^2(x) > \frac{1-\eps}{w(t)}$. Thus we can
  conclude that $\|\delta_t\|^*_D = \frac{1}{\sqrt{w(t)}}.$
\end{proof}

Let $(\eps_n)_{n=1}^\infty$ (with $\eps_1$ small!) be a strictly
decreasing sequence of positive real numbers converging fast to $0$. For
each $i \in \enn$ let us define a base of neighborhoods
$(D_{i,n})_{n=1}^\infty$ in $[0,1]$: Let $i=1$ and \[D_{1,1} =
[0,2^{-1} + \eps_1), D_{1,2} =
(2^{-1} - \eps_2,1].\] We call this the first level. For the second level put
\[D_{1,3} = [0,2^{-2} + \eps_3), D_{1,4} = (2^{-2} -
\eps_4,2\cdot 2^{-2} + \eps_4),\] \[D_{1,5} =
(2\cdot 2^{-2} - \eps_5,3\cdot 2^{-2} + \eps_5), D_{1,6} = (3\cdot 2^{-2} - \eps_6,1].\] Continue in this
fashion to obtain the base $(D_{1,n})_{n=1}^\infty$ consisting of
open intervals in $[0,1]$. Finally
let $D_i=(D_{i,n})_{n=1}^\infty$ be the base of $[0,1]$ consisting of the intervals in
$(D_{1,n})_{n=1}^\infty$ starting from level $i$.   

We will prove that for $i \ge 2$ the space $X_{D_i}$ fails to have the
SD2P. In fact, we will prove the following.

\begin{prop}\label{prop:snorm-nonSD2P}
 For each $i \ge 2$ let $X_{D_i}$ be the space $C[0,1]$ with the norm
 $\|\cdot\|_{D_i}$. Then for every $\eps > 0$ there exists finite convex combinations
 of slices of $B_{X_{D_i}}$ with diameter at most $\frac{\sqrt{i + \eps}}{i}$. 
\end{prop}

\begin{proof}
  First suppose $i=2$ and choose $t_1 = 0$ and $t_2=1$ and note that
  $J(t_1) \cap J(t_2) = \emptyset$, $\{t_1\} \cap \overline{D_n} =
  \emptyset$ for every $n \not\in J(t_1)$, and $\{t_2\} \cap \overline{D_n} =
  \emptyset$ for every $n \not\in J(t_2)$. Put $M=\sup\{\|x\|_\infty:
  x \in B_{X_{D_2}}\} < \infty$. By a similar argument as in the last
  part of the proof of Lemma \ref{fact:dirac-norm} it is
  possible to choose, for any $\eps >0$, a $\eta >0$ such that
  \begin{align*}
    \sum_{n \not\in J(t_1)} 2^{-n}(2M\|x\|_{2,n}+ \|x\|^2_{2,n}) <
    \eps/3, 
    \sum_{n \not\in J(t_2)} 2^{-n}(2M\|y\|_{2,n} + \|y\|^2_{2,n})  < \eps/3,
  \end{align*}
  and 
  \begin{align*}
    \sum_{n \not\in J(t_1) \cup J(t_2)} 2^{-n}(\|x\|^2_{2,n} + 2\|x\|_{2,n}\|y\|_{2,n} + \|y\|^2_{2,n})  < \eps/3.
  \end{align*}
  whenever $x$ and $y$ are elements in the slices
  $S(\delta_{t_1}/\|\delta_{t_1}\|^*_{D_1}, \eta)$ and
  $S(\delta_{t_2}/\|\delta_{t_2}\|^*_{D_2},\eta)$ of $B_{X_{D_2}}$,
  respectively.
  Now, if we put $h = \frac{1}{2}x + \frac{1}{2}y$ we get
  \begin{align*}
    2^2\|h\|^2_{D_2} &= \sum_{n=1}^\infty 2^{-n}\|x + y\|^2_{2,n}\\
                & \le \sum_{n \in J(t_1)}
                  2^{-n}(\|x\|^2_{2,n} +2\|x\|_{2,n}\|y\|_{2,n}+ \|y\|^2_{2,n})\\ &+\sum_{n \in J(t_2)}      2^{-n}(\|x\|^2_{2,n} +2\|x\|_{2,n}\|y\|_{2,n}+
                  \|y\|^2_{2,n})\\ &+ \sum_{n \not \in J(t_1) \cup
                                      J(t_2)} 2^{-n}(\|x\|^2_{2,n}
                                      +2\|x\|_{2,n}\|y\|_{2,n}+
                                      \|y\|^2_{2,n})\\ 
                & \le 2 + \eps.
  \end{align*}
  For an arbitrary $i \ge 2$ we can in $[0,1]$ choose $i$ points
  $(t_k)_{k=1}^i$ such that $J(t_j) \cap J(t_k) = \emptyset$ for any
  $j \not= k$ and such that $\{t_k\} \cap \overline{D_n} =
  \emptyset$ for every $n \not\in J(t_k)$. Using a similar argument as
  for $i=2$ we get that for any $\eps > 0$ there
  exists for every $k = 1, \ldots, i$ a slice $S(\delta_{t_k}, \eta)$ of
  $B_{X_{D_i}}$ such that the convex combination \[\sum_{k=1}^i
  \frac{1}{i}S(\delta_{t_k}, \eta)\] has diameter at most
  $\frac{\sqrt{i+\eps}}{i}$.  
\end{proof}

\begin{thm}
  The space $\ell_2- \bigoplus_{i=1}^\infty X_{D_i}$ is MLUR, has the D2P,
  the LD2P+, and has convex combinations of slices of arbitrary small
  diameter.
\end{thm}

\begin{proof}
  The properties of being MLUR, having the D2P, and having the LD2P+
  are all stable by taking $\ell_2$-sums (see
  \cite[Theorem~3.2]{MR3098474} and \cite[Theorem~3.2]{IK} for the
  latter two). Thus the space $\ell_2- \bigoplus_{i=1}^\infty X_{D_i}$
  has to possess all these properties as well since each $X_{D_i}$
  does. So, what is left to prove is that the unit ball of $\ell_2-
  \bigoplus_{i=1}^\infty X_{D_i}$ has finite convex combinations of
  slices with arbitrary small diameter. To this end let $Z = X_{D_i}
  \oplus_2 Y_i$ where $Y_i = \ell_2-\bigoplus_{k\not=i} X_{D_i}$. Let
  $x^*_i \in S_{X^*_{D_i}}$, $S_i(x^*_i,\delta)$ a slice of
  $B_{X_{D_i}}$, and let $0 < \delta < \eta$. Now, if
  $(x_i,y_i)$ is in the slice
  $S((x^*_i,0),\delta)$ of $B_Z$, then $x^*_i(x_i) > 1-\delta$, and
  so $\|x_i\| > 1 - \delta$. Thus $\|y_i\|^2 \le 2\delta - \delta^2$.
  But this means that \[S((x^*_i,0),\delta) \subset S_i(x^*_i,\delta) \times
  (2\delta-\delta^2)^{1/2}B_{Y_i}.\] From this we see that if $z \in
  \sum_{j=1}^i \frac{1}{i}S_j(B_Z,(x^*_{i,j},0),\delta)$, then we can
  write $z = x + y$ where $x \in \sum_{j=1}^i \frac{1}{i}
  S_{i,j}(B_{X_{D_i}},x^*_{i,j},\delta)$ and $y \in
  (2\delta-\delta^2)^{1/2}B_{Y_i}$. Now, if the convex combination
  $\sum_{j=1}^i \frac{1}{i} S_{i,j}(B_{X_{D_i}},x^*_{i,j},\delta)$ is
  chosen so that its diameter is at most $\frac{\sqrt{i + \eta}}{i}$
  , which is possible by Proposition \ref{prop:snorm-nonSD2P}, we get
  that $\|x\| \le \frac{\sqrt{i+\eta}}{i}$ and $y \in   (2\delta -
  \delta^2)B_{Y_i}$. As $\|z\| \le  \|x\| + \|y\|$ and $i$ can be
  chosen as big as desired and $\eta > 0$ as small as desired, we are done. 
\end{proof}

\section{The local diameter 2 property +}\label{sec:ld2p+}

Let $X$ be a Banach space and $I$ the identity operator on $X$. Recall
that $X$ has the \emph{Daugavet property} if the equation 
\begin{equation}
  \label{eq:3}
   \|I+T\| =  1 + \|T\| 
\end{equation}
holds for every rank 1 operator $T$ on $X$. The Daugavet property
can be characterized as follows (see \cite{MR1856978} or \cite{MR1784413}):

\begin{thm}\label{thm:daugavet-char}
Let $X$ be a Banach space. Then the following statements are
equivalent.
\begin{itemize}
  \item[a)]$X$ has the Daugavet property.
  \item[b)]The equation $\|I + T\|= 1 + \|T\|$ holds for every weakly
      compact operator $T$ on $X$.
  \item[c)] For every $\eps > 0$, every $x \in
   S_X$, and every $x^* \in S_{X^*}$, there exists $y \in S(x^*,\eps)$ such that
   $\|x + y\| \ge 2 - \eps$.
  \item[d)] For every $\eps > 0$, every $x^* \in
   S_{X^*}$, and every $x \in S_X$, there exists $y^* \in S(x,\eps)$ such that
   $\|x^* + y^*\| \ge 2 - \eps$.
  \item[e)] For every $\eps > 0$ and every $x \in S_X$ we have $B_X =
    \ncconv(\Delta_\eps(x))$ where $\Delta_\eps(x)=\{y \in B_X: \|y-x\|
    \ge 2-\eps\}$.
\end{itemize}  
\end{thm}

Let us recall from the Introduction the definition of the
LD2P+ and at the same time introduce its weak$^*$ version. 
 
\begin{defn}
  We say that a Banach space $X$ has the \emph{local diameter 2
    property +} (LD2P+) if for every $x^* \in S_{X^*}$, every
  $\eps > 0$, every $\delta > 0$, and every $x \in
  S(x^*,\eps) \cap S_X$ there exists $y \in S(x^*,\eps)$ with $\|x
  -y\| > 2-\delta$. If $X$ is a dual space and the above holds for
  weak$^*$ slices $S(x^*,\eps)$, then $X$ is said to have the
  \emph{weak$^*$ local diameter 2 property + (weak$^*$-LD2P+)}.  
\end{defn}

From \cite[Theorem~1.4]{IK} and \cite[Open problem (7)
p.~95]{MR1856978} the following is known.

\begin{thm}\label{thm:ikw}
  Let $X$ be a Banach space. Then the following statements are
  equivalent.
  \begin{enumerate}
    \item [a)] The equation $\|I - P\| = 2$ holds for every norm-1 rank-1
      projection $P$ on $X$. 
    \item [b)] For every $\eps>0$, every $x^* \in S_{X^*}$ and every
      $x \in S(x^*,\eps)$ there exists $y \in S(x^*,\eps) \cap S_X$
      with $\|x-y\| > 2-\eps.$
    \item [c)] For every $x \in S_X$ and every $\eps >0$ we have $x
      \in \ncconv(\Delta_\eps(x))$ where $\Delta_\eps(x) = \{y \in
      B_X: \|x - y\| > 2 - \eps\}.$
  \end{enumerate}
\end{thm}

From Lemma \ref{lem:subslice} of Kadets and Ivakhno (see
\cite[Lemma~2.1]{IK}) stated below it is clear that the LD2P+ is
equivalent to the statements in Theorem \ref{thm:ikw}. Therefore
every Daugavet space has the LD2P+. Note, however, that the converse
is not true as the LD2P+ is stable by taking unconditional sums of
Banach spaces which fails for spaces with the Daugavet property (see
e.g. \cite[Corollary~3.3]{IK}).

\begin{lem}[Kadets and Ivakhno]\label{lem:subslice}
  Let $\eps>0$ and $x^* \in S_{X^*}$. Then for every $x \in S(x^*,\eps) \cap
  S_X$ and every positive $\delta < \eps$
  there exist $y^* \in S_{X^*}$ such that  $x \in S(y^*,\delta)$ and
  $S(y^*,\delta) \subset S(x^*,\eps).$
\end{lem}

In the proof of Proposition \ref{prop:ld2p+-char} below we will need
the following weak$^*$-version of Lemma \ref{lem:subslice}. Its proof
is more or less verbatim to that of Lemma \ref{lem:subslice} and will
therefore be omitted.

\begin{lem}\label{lem:w*subslice}
  Let $\eps>0$ and $x \in S_X$. Then for every $x^* \in S(x,\eps) \cap
  S_{X^*}$ which attains its norm and every positive $\delta < \eps$
  there exist $y \in S_X$ such that  $x^* \in S(y,\delta)$ and
  $S(y,\delta) \subset S(x,\eps).$
\end{lem}

We will now add to the list of statements in Theorem \ref{thm:ikw}
statements similar to b) and d) in Theorem \ref{thm:daugavet-char}. 

\begin{thm}\label{prop:ld2p+-char}
  Let $X$ be a Banach space. Then the following statements are equivalent:
  \begin{enumerate}
    \item [a)]$X$ has the LD2P+.
    \item [b)]For every $x \in S_X$, every $\eps > 0$, every
      $\delta > 0$, and every $x^* \in S(x,\eps) \cap S_{X^*}$ there
      exists $y^* \in S(x,\eps)$ with $\|x^* -y^*\| > 2-\delta$.
    \item [c)]The equation $\|I - P\|= 1 + \|P\|$ holds for every weakly
      compact projection $P$ on $X$.
  \end{enumerate}
\end{thm}

\begin{proof}
  a) $\Rightarrow$ b). By the Bishop-Phelps
  theorem we can assume without loss of generality that $x^* \in
  S(x,\eps) \cap S_{X^*}$ attains its norm. Let $0 < \eta <
  \min\{\eps,\delta/2\}$ and find by Lemma  \ref{lem:w*subslice} $y \in
  S_X$ such that $x^* \in S(y,\eta)$ and $S(y,\eta) \subset
  S(x,\eps)$. Note that $y \in S(x^*,\eta)$ and thus,  since $X$ has
  the LD2P+, we can find $z \in S(x^*,\eta)$ such that  $\|y - z\| >
  2-\eta.$ Hence there is $y^* \in S_{X^*}$ such that \[y(y^*)-z(y^*)
  = (y-z)(y^*) > 2 - \eta.\] From this we have $y(y^*) > 1-\eta$ and
  $z(y^*) > 1 - \eta$. It follows that $y^* \in S(x,\eps)$ as $S(y,\eta) \subset
 S(x,\eps)$. Moreover, using that $z \in S(x^*,\eta)$ and b), we have
 \begin{align*}
   \|x^* - y^*\| &\ge (x^*-y^*)(z) \\ 
                        &= x^*(z) - y^*(z) \\ 
                        & > 1 -  \eta + 1 - \eta > 2 - \delta.
  \end{align*}
 
   b) $\Rightarrow$ a). The proof is identical to the proof of the converse except
  that one does not have to use the Bishop-Phelp's theorem and that
  one uses \cite[Lemma~2.1]{IK} in place of Lemma
  \ref{lem:w*subslice}.
 
   a) $\Rightarrow$ c). The proof is similar to that of 
   \cite[Theorem~2.3]{MR1621757}.

  c) $\Rightarrow$ a). This is clear as c) trivially implies a) in Theorem \ref{thm:ikw}. 
\end{proof}

Note that $c_0$ does not have the LD2P+ as $e_1 \in S(e_1,\eps) \cap
S_{c_0}$ for every $1 \ge \eps >0$, but every point in $S(e_1,\eps)$
is of distance 1 or less from $e_1$. $c_0$ is the prototype of an
M-embedded space. Since the dual is an
M-embedded space has the RNP (see e.g. \cite[III.3
Corollary~3.2]{HWW}) we actually get from Proposition
\ref{prop:ld2p+-char} that every M-embedded space fails the LD2P+.

\begin{cor}\label{cor:m-emb-ld2p+}
  M-embedded spaces fail the LD2P+. 
\end{cor}

It is known that all the diameter 2 properties in Definition \ref{defn:diam2p} as
well as the Daugavet
property are inherited by certain subspaces called ai-ideals (see
\cite{ALN2} and \cite{A3}). We will end this section by showing that
this is true for the LD2P+ as well.

A subspace $X$ of a Banach space $Y$ is called
an \emph{ideal} in $Y$ if there exists a norm 1 projection $P$ on $Y^*$  with
$\ker P = X^\perp $. $X$ being an ideal in $Y$ is in turn equivalent
to the $X$ being locally 1-complemented in $Y$, i.e. for every $\eps>0$ and
every finite-dimensional subspace $E\subset Y$ there exists $T:E\to X$
such that
\begin{itemize}
  \item [a)] $Te=e$ for all $e\in X\cap E$.
  \item [b)] $\|Te\|\leq (1+\eps)\|e\|$ for all $e\in E$.
\end{itemize}

Following \cite{ALN2} a subspace $X$ of a Banach space $Y$ is called
an \emph{almost isometric ideal (ai-ideal)} in $Y$ if $X$ is locally 1-complemented
with almost isometric local projections, i.e., for every $\eps>0$ and
every finite-dimensional subspace $E\subset Y$ there exists $T:E\to X$
which satisfies a) and 
\begin{itemize}
  \item [b')] $(1-\eps)\|e\|\leq\|Te\|\leq (1+\eps)\|e\|$ for all $e\in E$.
\end{itemize}

Note that an ideal $X$ in $Y$ is an ai-ideal if $P(Y^*)$ is a 1-norming subspace of
$Y^\ast$. Ideals $X$ in $Y$ for which $P(Y^*)$ is a 1-norming
subspace for $X$ are called \emph{strict ideals}. An ai-ideal is, however,
not necessarily strict (see \cite{ALN2}).

\begin{prop} \label{prop:ld2p+-ai-ideal}
  Let $Y$ have the LD2P+ and assume $X$ is an ai-ideal in $Y$. Then $X$
  has the LD2P+.
\end{prop}

\begin{proof} For $\delta>0$, $Z$ a subspace of $Y$, and $x \in S_Z$ put 
  \[\Delta_\delta^Z(x) = \{y \in B_Z: \|x-y\| > 2- \delta\}.\]
  
  Let $x\in S_X$, $\eps >0$, and $\alpha >0$. We will show that there
  exists $z \in \mbox{conv}\Delta_\eps^X(x)$ with $\|x-z\| <
  \alpha$. First, since $Y$
  enjoys the LD2P+, we know that for any positive $\beta < \eps$ and
  any positive $\gamma < \alpha$ we can find
  $y=\sum_{n=1}^N\lambda_ny_n \in \mbox{conv}\Delta_{\beta}^Y(x)$ with
  $(y_n)_{n=1}^N \subset \Delta_{\beta}^Y(x)$ 
  such that $\|x-y\| < \gamma$. Now
  let $E=\mbox{span}\{y_1, \ldots, y_N, x\}$ and pick a
  local projection $T:E\to X$ such that $T$ is a
  $(1+\eta)$-isometry with $\eta > 0$ so small that 
  $(1+\eta)\gamma + \eta <  \alpha$, and $(1-\eta)(2-\beta) - \eta > 2 -
  \eps$. Put $z_n=\frac{Ty_n}{\|Ty_n\|}$ and
  $z=\sum_{n=1}^N\lambda_nz_n$. As $Tx=x$ we get 
  \begin{align*}
    \|x-z\| &\le \|x-Ty\|+ \|Ty - z\|\\ 
                & \le \|T(x -y)\| + \sum_{n=1}^N\lambda_n\big|1 - \|Ty_n\|\big| \\
                &< (1+ \eta)\gamma + \max_{1 \le n \le N}{\big|1 -
                  \|Ty_n\|\big|}\\
                & \le (1+ \eta)\gamma  + \eta < \alpha.  
  \end{align*}
  
  Moreover for every $1 \le n \le N$ we have 
  \begin{align*}
    \|x - z_n\| & = \|T(x - \frac{y_n}{\|Ty_n\|})\| \\ & \ge (1-\eta)\|x -
                        \frac{y_n}{\|Ty_n\|}\|\\
                      & \ge (1-\eta)(\|x - y_n\| - \|y_n -
                        \frac{y_n}{\|Ty_n\|}\|)\\
                       & \ge (1-\eta)(2-\beta -
                         \frac{\big|1-\|Ty_n\|\big|}{\|Ty_n\|}\|y_n\|)\\
                       & \ge (1-\eta)(2-\beta -
                         \frac{\eta}{1-\eta}) > 2- \eps,
  \end{align*}
  Thus $(z_n)_{n=1}^N \subset \Delta_\eps(x)$ and as
  $\alpha > 0$ is arbitrarily chosen, we are done.
\end{proof}

\section{The difficulty of finding an MLUR norm on $c_0$ with the D2P}
This section is motivated by the question whether it is possible to
construct an MLUR norm on $c_0$ with the D2P. Actually this turns out to be
much harder than in $C[0,1]$. From Proposition \ref{fact:1} below we see
that if such a norm exists, it cannot be a lattice norm. 

\begin{defn}
  Let $X$ be a Banach lattice. A projection $P:X \to X$ is said to be
  a \emph{lattice projection}
  if $u$ and $v$ are disjoint whenever $u \in P(X)$ and 
  $v \in \ker P$, the kernel of $P$. 

  We say that \emph{the identity can be approximated by finite rank lattice
  projections at a point $x \in X$} if for all $\eps > 0$ there exists a
  lattice projection $P$ with finite rank such that $\|x - Px\| < \eps.$ 
\end{defn}

\begin{prop}\label{fact:1}
  Let $X$ be a Banach lattice and $x \in B_X$ a strongly extreme
  point. If the identity can be approximated by finite rank lattice
  projections at $x$, then $x$ is a denting point.     
\end{prop}

To prove this we will use the following lemma.

\begin{lem}\label{claim:1}
  Let $\eps >0$. Then there exists $\delta >0$ such that for any
  lattice projection $P:X \to X$, the condition $\|x - Px\| < \delta$,
  $u \in \ker P$, and $\|Px + u\| \le 1 + \delta$ imply $\|u\| < \eps$.
\end{lem}

\begin{proof}
  Since $x \in B_X$ is strongly extreme, there exists $\eta >0$ such
  that $\|y\| < \eps$ whenever $\|x \pm y\| \le 1 + \eta$. Put
  $\delta= \eta/3$ and suppose $P$, $x$, and $u$ satisfy the
  assumptions. As $P$ is a lattice projection
  we have 
  \[\|Px + u\| = \|Px - u\|.\]
  Thus
  \[\|x \pm u\|\le \|Px \pm u\| + \|x - Px\| \le 1 + 3\delta = 1 +
  \eta\] 
   and hence $\|u\| \le \eps.$
\end{proof}

\begin{cor}
  Let $\eps, \delta > 0$, let $P$ satisfy the assumptions in Lemma
  \ref{claim:1}, and let
  \[W = \{w \in X: \|P(x-w)\| < \delta\}.\]
 Then 
\[\mbox{diam}(W \cap B_X) < 2\eps + \delta.\]
\end{cor}

\begin{proof}
  Pick $w \in W \cap B_X$ and put $u=w - Pw$. Then
 \[\|Px + u\| = \|P(x-w)\| + \|w\| \le 1 + \delta.\]
 From Lemma \ref{claim:1} we get $\|u\| < \eps$. Thus
 \begin{align*}
   \|x-w\|&= \|x+u-Pw\| \\&\le \|P(x-w)\| + \|x-Px\| + \|u\| \le 2\delta
   + \eps,
 \end{align*}
and so we are done.
\end{proof}

\begin{proof}[Proof of Proposition \ref{fact:1}]
  If $\dim(PX) < \infty$, we get that $W$ is weak open. This implies
  that $x$ is a point of continuity for $B_X$. Since every point of
  weak to $\|\cdot\|$ continuity which is extreme is a denting point
  \cite[Theorem]{LLT},  we get that $x$ is denting. 
\end{proof}

\section{Questions}
Let us end the paper with some questions that is suggested by the
current work: 

\begin{quest}
  Does there exist an equivalent MLUR norm on $c_0$ with LD2P?
\end{quest}

\begin{quest}
Does there exist a Banach space
with the LD2P and which is weakly locally uniformly rotund?
\end{quest}

Regarding this question we note that there does exist a Banach space $X$ which
is weakly uniformly rotund (wUR) and which has the property that for
every $\eps >0$ and every weak$^*$ null sequence $(f_n) \subset
S_{X^*}$ the diameter of the slices $S(f_n,\eps)$ tends to 2. Such a 
Banach space can be constructed as follows: Let $1<p_1 < p_2 < \ldots$ a
sequence such that 
\begin{equation}\label{eq}
\prod_{i\in \enn} \|\mathrm{I} \colon \ell_\infty (2) \to \ell_{p_{i}} (2)\| < 2 
\end{equation} 
(operator norms of the formal identity mappings between
$2$-dimensional $\ell_p$ spaces). Then one can form a Banach sequence
space as follows: \[X = \err \oplus_{p_1} (\err \oplus_{p_2} (\err \oplus_{p_3} (\ldots \ \ldots )))\] 
where $\err$ is considered a $1$-dimensional Banach space and the
space is normed by first defining semi-norms in finite-dimensional initial parts according to the above schema and then 
taking a limit of the semi-norms in a similar way as in the construction of the
variable exponent spaces introduced in \cite{MR2852873}.  We will now
show that this space $X$ has the above mentioned properties. 
 
\begin{proof}
Put $Y =\overline{\spann} (e_n \colon n\in\enn ) \subset X$
and $Y_k := \overline{\spann} (e_n \colon n\in\N,\ n\geq k ) \subset
X.$ It can be seen from arguments in \cite{MR2852873} that $X$ and $Y$ are
isomorphic to $\ell_\infty$ and $c_0$ respectively. Also the tail
spaces $Y_k$ become asymptotically isometric to $c_0$, i.e. for each $\varepsilon>0$ there is $k\in \enn$ such that the tail spaces $\Y_j$, $j\geq k$, are  $1+\varepsilon$-isomorphic to $c_0$ via a linear mapping which identifies the 
canonical unit vector bases of $Y_j$ and $c_0$.

The wUR part. Let $(x_n ), (y_n) \in B_{Y}$ be such that 
$\|x_n + y_n \|_Y \to 2$. Denote by $P_n$ the basis projection to the first $n$-coordinates and let
$Q_n  = \mathrm{I} - P_n$ be the coprojection to the rest of the coordinates. Then according to the 
definition of the space 
\begin{equation}\label{eq2} 
\left(|P_{1} (x_n + y_n ) |^{p_1} + \|Q_1 (x_n + y_n )\|^{p_1}\right)^{\frac{1}{p_1}}\to 2,
\end{equation}
so by the triangle inequality
\[\left((|P_{1} (x_n )|+ |P_1 (y_n ) |)^{p_1} + (\|Q_1 (x_n) \|+\|Q_1 (y_n )\|)^{p_1}\right)^{\frac{1}{p_1}}\to 2,\]
and by the uniform convexity of $\ell_{{p_1}} (2)$ we get that 
\begin{equation}\label{eq3}
|P_{1} (x_n )| - |P_1 (y_n ) | \to 0,\quad \|Q_1 (x_n) \|-\|Q_1 (y_n )\|\to 0.
\end{equation}
By inspecting \eqref{eq2} we obtain $|P_{1} (x_n - y_n ) |\to 0$.
By continuing inductively, using the right-hand side of \eqref{eq3}, we get that 
$P_k (x_n - y_n ) \to 0$ for each $k$. Recall that $Y$ is isomorphic to $c_0$, thus $Y^*$ is isomorphically $\ell_1$. Therefore $x_n -y_n \to 0$ weakly.

The large slices part.  First note that if $(f_n ) \subset Y^*$ is a normalized sequence then 
$\|f_n  \|_{\ell^1}\geq 1$ because $\|\cdot\|_{c_0} \leq \|\cdot \|_{Y}$. 
Fix $\varepsilon >0$. Let $k \in \enn$ be such that 
\[\sum_{i=1}^{\infty} a_{i} e_{k+i} \mapsto \sum_{i=1}^{\infty} a_{i} e_{i}\]
defines a $(1+\varepsilon/4)$-isomorphism $Y_k \to c_0$. Note that then 
\[\frac{1}{(1+\varepsilon/4)} \|f \circ Q_k\|_{\ell_1} \leq\|f \circ Q_k \|_Y \leq  (1+\varepsilon/4) \|f \circ Q_k\|_{\ell_1},\quad f\in \ell_1 .\]

Because $(f_n )$ is weak-star null we may choose $m_0 \in \enn$ such that sufficiently large part 
of the mass is supported on the domain of $Q_k$, more precisely, 
\[\frac{1-\varepsilon}{\|f_m \circ Q_k \|_{\ell_1}} < \frac{1}{1+\varepsilon/3}\]
for $m\in \enn$, $m\geq m_0$.

Put $g=\frac{f_m \circ Q_k}{\|f_m \circ Q_k\|_{\ell_1}}$. Then 
\[\left\{x\in c_0 \colon g(x)>\frac{1}{1+\varepsilon/3}\right\}\subset \left\{x\in c_0 \colon (f_m \circ Q_k)(x) >1-\varepsilon\right\}.\]

Note that $\frac{1}{(1+\varepsilon/4)}B_{c_{0}}\cap Y_k \subset
B_{Y_k}$. Therefore the above inclusion yields that we may pick 
\[x,y \in \{z\in B_{Y_k}\colon f_m (z) >1-\varepsilon\}\]
with 
\[\|x-y\|_Y \geq \|x-y\|_\infty > \frac{2}{(1+\varepsilon/3)}.\]
\end{proof}

\begin{quest}
  Does there exist a Banach space with the LD2P+ which fails the D2P?
\end{quest}

\begin{quest}\label{quest:3}
  Does every Banach space with the LD2P+ contain a copy of $\ell_1$?
\end{quest}

\def\cprime{$'$} \def\cprime{$'$}
\providecommand{\bysame}{\leavevmode\hbox to3em{\hrulefill}\thinspace}
\providecommand{\MR}{\relax\ifhmode\unskip\space\fi MR }
\providecommand{\MRhref}[2]{%
  \href{http://www.ams.org/mathscinet-getitem?mr=#1}{#2}
}
\providecommand{\href}[2]{#2}

\end{document}